\documentclass[]{interact}
\usepackage{hyperref}
\usepackage{epstopdf}%
\usepackage[caption=false]{subfig}%

\usepackage[numbers,sort&compress]{natbib}%
\bibpunct[, ]{[}{]}{,}{n}{,}{,}%
\renewcommand\bibfont{\fontsize{10}{12}\selectfont}%

\theoremstyle{plain}%
\newtheorem{thm}{Theorem}[section]
\newtheorem{lem}[thm]{Lemma}
\newtheorem{coro}[thm]{Corollary}
\newtheorem{prop}[thm]{Proposition}

\theoremstyle{definition}
\newtheorem{defn}[thm]{Definition}

\newtheorem{assumption}[thm]{Assumption}

\theoremstyle{remark}
\newtheorem{remark}[thm]{Remark}

\usepackage{amsmath,amssymb,amsthm,amsfonts,mathrsfs,mathtools,enumerate,mathtools,color,bbm,ifthen,graphicx,overpic}
\usepackage[shortlabels]{enumitem}
\usepackage[french,english]{babel}
\usepackage[dvipsnames]{xcolor}

\newcommand{\1}{\mathbbm{1}}
\newcommand{\R}{\mathbb{R}}
\renewcommand{\P}{\mathbb{P}}
\newcommand{\E}{\mathbb{E}}

\newcommand{\cF}{\mathcal{F}}

\newcommand{\mcB}{\mathscr{B}}

\newcommand{\cC}{\mathcal{C}}
\newcommand{\cP}{{\mathcal P}}

\newcommand{\cK}{\mathcal{K}}

\newcommand{\sH}{H}
\newcommand{\sx}{\mathsf{x}}

\renewcommand{\geq}{\geqslant}
\renewcommand{\leq}{\leqslant}
\DeclareMathOperator{\supp}{supp}

\DeclareMathOperator{\conv}{conv}
\DeclareMathOperator{\graph}{graph}
\DeclareMathOperator{\cl}{cl}
\DeclareMathOperator{\zer}{zer}
\DeclareMathOperator{\BC}{BC}
\DeclareMathOperator{\essacc}{ess\,acc}
\DeclareMathOperator*{\argmax}{arg\,max}

\newcommand{\ps}[1]{\langle #1\rangle}

\begin{document}

\articletype{}%

\title{A closed-measure approach to stochastic approximation}

\author{\name{P. Bianchi\textsuperscript{a}\thanks{Email: bianchi@telecom-paris.fr}
    and R. Rios-Zertuche\textsuperscript{b}}
  \affil{\textsuperscript{a}LTCI, Telecom Paris, Institut Polytechnique de Paris\\
    \textsuperscript{b}UiT The Arctic University of Norway}%
}

\maketitle

\begin{abstract}
  This paper introduces a new method to tackle the issue
  of the almost sure convergence of stochastic approximation
  algorithms defined from a differential inclusion.  Under the
  assumption of slowly decaying step-sizes, we establish that the set
  of essential accumulation points of the iterates belongs to the
  Birkhoff center associated with the differential inclusion.  Unlike
  previous works, our results do not rely on the notion of asymptotic
  pseudotrajectories introduced by Bena\"im--Hofbauer--Sorin, which is the
  predominant technique to address the convergence problem. They follow
  as a consequence of Young's superposition principle for closed
  measures.  This perspective bridges the gap between Young's principle and
  the notion of invariant measure of set-valued dynamical systems
  introduced by Faure and Roth.
  Also, the proposed method allows to obtain
  sufficient conditions under which the velocities
  locally compensate around any essential accumulation point.
\end{abstract}


\section{Introduction}
Let $H\colon \R^n\rightrightarrows\R^n$ be an upper semi-continuous set-valued map, and let $\pi\colon\R^n\times\R^n\to\R^n$ be the projection $\pi(x,v)=x$.

In this paper we are concerned with stochastic processes of the following form:
\[x_{i+1}=x_i+\epsilon_i\theta_i+\epsilon_i\eta_{i+1},\]
where $\theta_i$ is in a $\delta_i$-neighborhood of the set $H(x_i)$, both sequences $(\delta_i)\subset[0,+\infty)$ and $(\eta_{i})\subset \R^n$ are random, $\delta_i\to0$ as $i\to+\infty$, and $\eta_{i+1}$ is a martingale increment.  We assume that the step size sequence $(\epsilon_i)$ converges to 0; contrary to \cite{bor-livre08,ben-hir-96} but similarly to \cite{ben-sch-00,fau-rot-13}, we allow this convergence to happen arbitrarily slowly, hence covering a broad set of practical situations. The sequences $(x_i)$ model a discrete process with drift $\theta_i$, noise $\eta_i$, and step-size $\epsilon_i$; we will describe some of the situations where these arise in Section %
\ref{sec:app} below. 
An important example is the classical Robbins-Monro algorithm \cite{robbins1985stochastic}, in which $\theta_i=h(x_i)$ for some vector field $h$; when $h$ is Lipschitz continuous, the so-called ordinary differential equation method can be used to characterize the set of
accumulation points of $(x_i)$ \cite{ljung1977analysis,bor-livre08,kus-yin-(livre)03,ben-(cours)99}, and the analysis of more general versions of the algorithm require more sophisticated techniques \cite{ben-hir-96}.

We investigate the dynamics of the sequences $(x_i)$ using the method of ``closed measures,'' introduced in \cite{bolte2020long} in a more restrictive context.
Let us briefly describe the approach. We consider the probability measures 
\[\mu_i = \frac{\sum_{j=0}^i \epsilon_j\delta_{(x_j,\theta_j+\eta_{j+1})}}{\sum_{j=0}^i \epsilon_j}\]
on $\R^n\times\R^n$.
These are a sort of ``occupation measures'' of the sequence $(\mu_i)$ in the sense that they encode the ``position'' $x_j$ and ``velocity'' $\theta_j+\eta_{j+1}=(x_{i+1}-x_i)/\epsilon_i$ of the sequence $(x_i)$. We then consider the set $\mathcal A$ of weak* accumulation points of the sequence $(\mu_i)$. The elements $\mu\in\mathcal A$ are probability measures that encode the long-term, recurring dynamics of $(x_i)$. With mild assumptions \ref{hyp:H}, \ref{hyp:noise}, and \ref{hyp:step}, $\mathcal A$ is non-empty and the measures $\mu$ it contains have the interesting property of being closed, that is, the integrals of gradients vanish: for all $\phi\in C^\infty(\R^n)$, we have 
\[\int_{\R^n\times\R^n}\langle\nabla\phi(x),v\rangle\,d\mu(x,v)=0.\]
The Young superposition principle, Th. \ref{prop:superposition}, is available for measures with this closedness property; its statement implies that the measures $\mu\in\mathcal A$ can be understood as being composed of occupation measures of solutions of the differential inclusion $H$ in a sense that is made precise in Definition \ref{def:invariantmeasure}. In particular, this means that the measures $\mu$ are invariant for $H$, and that the projections $\pi(\supp\mu)$ of their supports are contained in the Birkhoff center of $H$, defined as the closure of all the recurrent points of the solutions of the differential inclusion $x'(t)\in H(x(t))$; in this way, we recover results of \cite{fau-rot-10,fau-rot-13,ben-sch-00}, formerly obtained using the theory of asymptotic pseudo-trajectories \cite{ben-hof-sor-05,ben-hof-sor-(partII)06}. 

The projections $\pi(\supp\mu)$, $\mu\in\mathcal A$, are also shown to compose the essential accumulation set $\essacc (x_i)$  of the sequence $(x_i)$; this is a subset of the accumulation set in which we know that $(x_i)$ spends substantial time. We thus clarify the link between $\essacc(x_i)$ and the Birkhoff center of $H$. 

We also focus on the stable zeros $\zer_s(H)$ of $H$, which are defined to be the points $x\in \R^n$ such that $0\in H(x)$ and the only solution of the differential inclusion $x'(t)\in H(x(t))$ is $x(t)=x$. We are able to show, roughly speaking, that the average of the ``velocities'' $v_i\coloneqq\theta_i+\eta_{i+1}$ of $(x_i)$ vanishes on $\zer_s(H)$. Moreover, this vanishing implies the oscillation compensation property: if $\zer_s(H)=\BC(H)$, then for all bounded, continuous $\psi\colon\R^n\to(0,+\infty)$ we have (for some subsequence $i_k$)
\[\lim_{k\to+\infty}\frac{\sum_{j\leq i_k}\epsilon_j\psi(x_j)v_{j}}{\sum_{j\leq i_k}\epsilon_j\psi(x_j)}=0,\]
as long as $\liminf_{k}\sum_{j\leq i_k}\epsilon_j\psi(x_j)/\sum_{j\leq i_k}\epsilon_j>0$.
Using $\psi$ to approximate the indicator function of a ball $B$, the compensation of the oscillations can be seen to be roughly equivalent (up to some technical assumptions) to saying that, on any ball $B$ centered at
    any point $x\in\essacc(x_i)$, the weighted average of the velocities $v_i$,
    \[\frac{\sum_{\substack{j\leq i\\x_j\in B}}\epsilon_jv_j}{\sum_{\substack{j\leq i\\x_j\in B}}\epsilon_j}=\frac{\sum_{\substack{j\leq i\\x_j\in B}}\epsilon_j(\theta_j+\eta_{j+1})}{\sum_{\substack{j\leq i\\x_j\in B}}\epsilon_j}=\frac{1}{\sum_{\substack{j\leq i\\x_j\in B}}\epsilon_j}\sum_k(x^+_k-x^-_k),\]
    vanishes asymptotically; here we have denoted $x^+_k$ and $x^-_k$ the points at which the sequence $(x_i)$ enters and exits the ball $B$ for the $k$-th time, respectively. Observe that the asymptotic vanishing of this quantity implies a ``slow down'' of the sequence $(x_i)$ (i.e., if $x^+_k=x_{i+N}$ and $x^-_k=x_i$, the distance traversed $\|x^+_k-x^-_k\|=\|x_{i+N}-x_i\|$ grows smaller with respect to the time $\sum_{j=i}^{i+N}\epsilon_j$ this takes), at least when we know that $\|x^+_k-x^-_k\|$ is uniformly bounded away from 0. This generalizes some of the results of \cite{bolte2020long}.
    
 The framework considered in this paper fits, in particular, the case in which $H$ is the Clarke subdifferential of a Lyapunov function $V$; see Corollary \ref{cor:compens}. Other interesting applications are presented in Section \ref{sec:app}, and include  the stochastic descent and heavyball algorithms, and the best-response dynamics in games.

The plan of the paper is the following. %
Section~\ref{sec:prelim}
  presents known facts about differential inclusions and subdifferentials.
  Section~\ref{sec:closed} introduces Young's superposition principle for closed measures,
  along with useful consequences.
  Section~\ref{sec:main} gives the main results of the paper and is devoted to the development of the analysis of the long-term dynamics of the iterates and their convergence. %
  Section~\ref{sec:app} gives some applications to optimization and game theory.

\section{Preliminaries}
\label{sec:prelim}

\subsection{Notations}

If $(E,d)$ is a Polish space equipped with its Borel $\sigma$-field, we denote by
$\cP(E)$ the set of probability measures on $E$. We denote by $\supp(\nu)$ the support of a measure $\nu$.
We denote by $\cC(I,E)$ the set of continuous functions
on $I\to E$, where $I$ is a real interval. The set $\cC(I,\R^n)$ (where $n$ is an integer) will always be equipped with
the topology of uniform convergence on compact intervals of $I$.
We denote by $\cC_b(\R^n,\R)$ the set of bounded continuous functions on $\R^n\to\R$, by $\cC^\infty(\R^n,\R)$
the set of infinitely differentiable functions on $\R^n\to\R$. 
The notation $\1_A$ represents the indicator function of a set $A$, equal to one on that set, zero otherwise.
The closure of a set $A$ is denoted by $\cl(A)$. 
For every set $A\subset\R^n$, define $A^\bot = \{u : \langle{u,a}\rangle=0,\, \forall a\in A\}$.
We denote by $d(x,A) = \inf\{d(x,y):y\in A\}$ the distance
between a point $x\in E$ and a non-empty subset $A\subset E$. 
We say that $X$ is a \emph{set-valued map} on $E$ into $E'$,
denoted by $X:E\rightrightarrows E'$, if $X(x)$ is a subset of $E'$ for every $x\in E$.
A point $x\in E$ is called a \emph{zero} of $X$ if $0\in X(x)$.
We denote by $\zer(X)$ the set of zeroes of $X$.
We denote by $\graph{(X)}\coloneqq\{(x,y):y\in X(x)\}$ the graph of $X$.
We say that $X$ is \emph{upper semicontinuous} (u.s.c.) if for every $x\in E$ and every neighborhood
$V$ of $X(x)$, there exists a neighborhood $U$ of $x$ such that $X(U)\subset V$.
The projection onto position space will be denoted $\pi\colon \R^n\times\R^n\to\R^n$, $\pi(x,v)=x$.  

  \subsection{Differential inclusions}

Let $\sH\colon{} \R^n\rightrightarrows  \R^n$ be a u.s.c. set-valued map
with non-empty compact convex values. We say that an absolutely continuous mapping $\gamma\colon{}\R_+\to \R^n$ 
is a solution to the differential inclusion
  \begin{equation}
    \gamma'(t) \in H(\gamma(t)),
    \tag{DI}
    \label{eq:di}
  \end{equation} 
with initial condition $x\in \R^n$,
if $\gamma(0)=x$ and if (\ref{eq:di}) holds for almost every $t\in \R_+$.
We denote by $S_\sH(x)$ the subset of $\cC(\R^+,\R^n)$ of all solutions to (\ref{eq:di}) issued from~$x$.
In the particular case, where $\sH$ satisfies the following linear growth condition:
\begin{equation}
  \label{eq:lin-growth}
\exists C>0,\ \forall x\in \R^n,\  \|\sH(x)\|\leq C(1+\|x\|)\ ,
\end{equation}
the set $S_\sH(x)$ is non-empty for every $x\in \R^n$ \cite{aub-cel-(livre)84}.

The limit set of a function $\gamma \in C(\R_+, \R^n)$ is defined as 
\begin{equation}
L(\gamma) \coloneqq \bigcap_{t\geq 0}\cl{(\gamma([t,+\infty))} \,.\label{eq:limit-set}
\end{equation}
It coincides with the set of points of the form $\lim_n \gamma(t_n)$ for some 
sequence $t_n\to\infty$.
\begin{defn}\label{def:birkhoffcenter}
  A point $x\in \R^n$ is said \emph{recurrent} w.r.t (\ref{eq:di}) if
  there exists $\gamma\in S_\sH(x)$ such that $x\in L(\gamma)$.
  The \emph{Birkhoff center}, denoted by $\BC(\sH)$, is the closure of the set of all recurrent points:
$$
\BC(\sH) \coloneqq \cl\left\{x\in \R^n\,:\,\exists \gamma\in S_\sH(x),\, x\in L(\gamma)\right\}\,.
$$
\end{defn}

\begin{defn}
\label{def:lyapunov}
Given $\Lambda \subset \R^n$, we say that a continuous function $V\colon{}\R^n\to \R$ is a \emph{Lyapunov function} for $\Lambda$
if for all $\gamma\colon{}\R_+\to\R^n$ solution to~(\ref{eq:di}), we have for all $t > 0$:
\begin{align*}
    \begin{cases}
    V(\gamma(t)) \leq V(\gamma(0)) & \\
    V(\gamma(t)) < V(\gamma(0)),& \text{ if } \gamma(0) \not\in \Lambda\,.
    \end{cases}
\end{align*}
\end{defn}
\begin{prop}
  \label{prop:lyap-BC}
  Let $\sH\colon{} \R^n\rightrightarrows  \R^n$ be a u.s.c.\ set-valued map
  with non-empty compact convex values. Consider $\Lambda\subset\R^n$ and let $V\colon{}\R^n\to \R$ be a Lyapunov function for $\Lambda$.
  Then, $\BC(H)\subset\Lambda$.
\end{prop}
\begin{proof}
Consider $x\in \BC(H)$ and $\gamma\in S_H(x)$ such that $x\in L(\gamma)$.
  Assume by contradiction that $x\notin \Lambda$. Then, $V(\gamma(1))<V(x)$. Set $\varepsilon \coloneqq V(x)-V(\gamma(1))$.
  As $V$ is continuous, there exists $\delta>0$ such that
  for every $y$, $\|y-x\|<\delta$ implies $V(y)\geq V(x) - \varepsilon/2$. As $x\in L(\gamma)$, there
  exists $t>1$ such that $\|\gamma(t)-x\|<\delta$. Thus, $V(\gamma(t))\geq V(x) -\varepsilon/2$.
  As $V(\gamma(t))\leq V(\gamma(1))$, we obtain a contradiction.
\end{proof}

\subsection{Clarke subdifferential}

Let $f\colon{}\R^n\to \R$ be a locally-Lipschitz continuous function. Denote by $\text{Reg}(f)$ the set of its differentiability points.
Recall that $\text{Reg}(f)$ is dense in $\R^n$ by Rademacher's theorem.
\begin{defn}[Clarke subdifferential]
  The Clarke subdifferential of $f$ at a point $x\in \R^n$ is the set
  $$
  \partial f(x)\coloneqq\overline{\conv}\{v\in \R^n:\exists (y_i)_i\in \text{Reg}(f)^{\mathbb N},\, y_i\to x\text{ and }\nabla f(y_i)\to v\}\,.
  $$
\end{defn}
The corresponding set-valued map $\partial f$ is u.s.c. with non-empty compact convex values.
\begin{defn}[Path differentiability]
  A locally-Lipschitz function $f\colon{}\R^n\to \R$ is said \emph{path-differentiable} if, for every
  locally-Lipschitz curve $\gamma\colon{}\R\to\R^n$, the composition $f\circ\gamma$ is differentiable at almost every point $t$, and satisfies:
  $$
  (f\circ\gamma)'(t) = \ps{v,\gamma'(t)}
  $$
  for every $v\in \partial f(\gamma(t))$. 
\end{defn}

\section{Closed measures}
\label{sec:closed}

\subsection{The superposition principle}
\label{sec:superposition}

Let $n$ be an integer.

Consider a probability measure $\mu\in\cP(\R^n\times\R^n)$.
Let $\pi\colon{}{\mathbb R}^n\times{\mathbb R}^n\to{\mathbb R}^n$ be the projection, $\pi(x,v)\coloneqq x$.
The disintegration theorem states that there exists a probability transition kernel
$\R^n\times\mcB(\R^n)\to [0,1]$, denoted by $(x,A)\mapsto \mu_x(A)$ such that for every
bounded continuous function $\varphi\colon\R^n\times\R^n\to \R$,
\begin{equation}
\mu(\varphi) = \int \left(\int \varphi(x,v) d\mu_x(v)\right) d(\pi_*\mu)(x)\,,\label{eq:disintegration}
\end{equation}
where $\pi_*\mu$ is the pushforward measure of $\mu$ through $\pi$ \textit{i.e.}, $\pi_*\mu(A)=\mu(A\times \R^n)$ for every $A\in \mcB(\R^n)$.
Note that:
$$
\supp (\pi_*\mu) = \pi(\supp \mu)\,.
$$
We shall say that $\mu$ has the \emph{disintegration}
$\mu = \int \mu_x d(\pi_*\mu)(x)$ if Eq.~(\ref{eq:disintegration}) holds for every bounded
continuous function $\varphi$. Whenever it is well defined, we will refer to the quantity
$$
v_\mu(x) \coloneqq \int vd\mu_x(v)\,,
$$
as the \emph{centroid field} associated with $\mu$.

\begin{defn}
  A compactly supported probability measure $\mu\in\cP(\R^n\times\R^n)$ is said \emph{closed}, if
  for every $g\in \cC^\infty(\R^n,\R)$,
  $$
  \int \ps{\nabla g(x),v}d\mu(x,v) = 0\,.
  $$
\end{defn}

We define the shift operator $\Theta\colon\cC(\R,\R^n)\to
\cC(\R,\cC(\R,\R^n))$ s.t. for every $\gamma\in \cC(\R,\R^n)$,
$\Theta(\gamma) \colon t\mapsto \gamma(t+\cdot)$.
A measure $\vartheta\in \cP(\cC(\R,\R^n))$ is said $\Theta$-invariant
if $\vartheta = \vartheta \Theta_t^{-1}$ for all $t$, where  $\Theta_t(\gamma) \coloneqq  \Theta(\gamma)(t)$.
\begin{thm}[Young's superposition principle]
  \label{prop:superposition}
  Let $\mu\in \cP(\R^n\times\R^n)$ be a closed probability measure.
  There exists $\vartheta\in \cP(\cC(\R,\R^n))$ such that:
  \begin{enumerate}[i)]
\item $\vartheta$ is $\Theta$-invariant.
\item $\vartheta$-almost every curve $\gamma$ is Lipschitz continuous and differentiable at 0. 
\item For every bounded measurable function $\varphi\colon{}\R^n\times\R^n\to \R$,
 \begin{equation}\label{eq:superposition}
  \int \varphi(x,v_\mu(x))d(\pi_*\mu)(x) = \int \varphi(\gamma(0),\gamma'(0))d\vartheta(\gamma)\,.
\end{equation}
\end{enumerate}
In addition, $\vartheta$-almost every $\gamma$ satisfies $\gamma'(t) = v_\mu(\gamma(t))$ for almost every $t\in \R$.
\end{thm}
\begin{proof}
The statement can be established as a consequence of \cite[Th. 6.2]{patrick}.
In order to make the paper self-contained, we provide a direct and constructive proof, which is inspired of the proof of \cite[Th. 8.2.1]{ambrosiogiglisavare}.

 \noindent \textbf{Smooth case.} Consider a closed measure $\mu\in \cP(\R^n\times\R^n)$ such that $v_\mu(x)$ is a smooth function.
Consider a compact set $U\subset \R^n$ containing the supports of $\rho$ and $v_\mu$. Consider the flow $\Phi\colon{}U\times\R\to U$ associated 
to the ODE $x'(t) = v_\mu(x(t))$, so that 
\[\Phi(x,0)=x\qquad\text{and}\qquad \frac{\partial\Phi}{\partial t}(x,t)=v_\mu(\Phi(x,t)),\] and denote $\Phi_t(x)=\Phi(x,t)$.
Consider any smooth function $\varphi\colon{}\R^n\to \R$. Define $g(t,x) = \varphi(\Phi_t(x))$.
Obviously, $s\mapsto g(t+s,\Phi_{-s}(x))$ is constant. Computing the derivative of this function at the point $s=0$, we obtain
the transport equation:
$$
0 = \partial_t g(t,x) - \ps{\nabla_x g(t,x), v_\mu(x)}\,.
$$
Integrating w.r.t. $\pi_*\mu$,
$$
\frac d{dt}\int g(t,x)d\pi_*\mu(x) =\int \ps{\nabla_x g(t,x), v_\mu(x)}d\pi_*\mu(x)
= \int \ps{\nabla_x g(t,x), v}d\mu(x,v)=0\,,
$$
as $\mu$ is closed. Therefore,
$\int \varphi(\Phi_t(x))d\pi_*\mu(x)$ is constant w.r.t. $t$.
In other words, $\pi_*\mu$ is an invariant measure of $\Phi$.

We define  $\vartheta$ as the probability measure on $\cC(\R,\R^n)$,
given for every $g\colon{}\cC(\R,\R^n)\to\R_+$ by: 
$$
\vartheta(g) \coloneqq  \int g(\Phi(x,\cdot))d\pi_*\mu(x)\,.
$$
As $\Phi$ preserves $\pi_*\mu$, it follows that $\vartheta$ is $\Theta$-invariant.
Let $\phi(x,v)$ be a bounded continuous function, and set $g(\gamma) = \phi(\gamma(0),\gamma'(0))$.
Note that $g$ is well-defined $\vartheta$-a.e., because $\vartheta$-almost every $\gamma$ is differentiable at 0. %
Moreover, $\gamma'(0)=v_\mu(\gamma(0))$ almost everywhere.
By definition of $\vartheta$, Eq.~(\ref{eq:superposition}) follows.

\noindent {\bf General case}.
We follow the same steps as in the proof of \cite[Th. 8.2.1]{ambrosiogiglisavare}.
Consider a closed measure $\mu$. 
Consider a smooth and compactly supported $\psi:\R^n\to\R_+$ such that $\int \psi(x)dx=1$.
For every $\varepsilon>0$, set $\psi_\varepsilon(x)=\psi(x/\varepsilon)/\varepsilon^n$. Define for every $x\in \R^n$,
$$
\rho_\varepsilon(x) \coloneqq 
    \int \psi_\varepsilon(x-y)\pi_*\mu(dy)\quad\text{and}\quad v_\epsilon(x)\coloneqq\begin{cases}\rho_\varepsilon(x)^{-1}\int \psi_\varepsilon(x-y) v d\mu(y,v),&
    \rho_\varepsilon(x)\neq 0,\\
    0,&\rho_\varepsilon(x)=0.
\end{cases}
$$
It is an easy exercise to verify that the measure defined by:
$$
\mu_\varepsilon(dx,dv) := \rho_\varepsilon(x)dx\otimes \delta_{v_\varepsilon(x)}(dv)\,
$$
is closed, and satisfies the smoothness assumption of the previous paragraph. Thus, there exists a $\Theta$-invariant measure $\vartheta_\varepsilon$ on $C(\R,\R^n)$ satisfying Eq.~(\ref{eq:superposition}). 
In particular, setting $\text{ev}_t (\gamma) := \gamma(t) $
for every $\gamma\in C(\R,\R^n)$, it holds that, for all $t$,
\begin{equation}
    \label{eq:ev-epsilon}
(\text{ev}_t)_*\vartheta_\varepsilon = \rho_\varepsilon(x)dx\,.
\end{equation}

We prove that the family $(\vartheta_\varepsilon: 0<\varepsilon<1)$ is tight in $\cP(C(\R,\R^n))$. 
As $\pi_*\mu$ and $\psi$ are compactly supported, there exists a compact set $K\subset\R^n$ such that, for all $\varepsilon\in(0,1)$ and for $\vartheta_\varepsilon$-almost every $\gamma$, 
$\gamma(0)\in K$. Moreover, for every $T>0$,
\begin{align}
\int\left(\int_0^T \|\gamma'(t)\|^2 dt\right)d\vartheta_\varepsilon(\gamma) &=     
T\int \|v_\varepsilon(\gamma(0))\|^2 d\vartheta_\varepsilon(\gamma) \nonumber \\
\notag &= T \int \|v_\varepsilon(x)\|^2 \rho_\varepsilon(x)dx \\
&\leq T \int \|v_\mu(x)\|^2 \pi_*\mu(dx)\,, \label{eq:tightvartheta}
\end{align}
where the last inequality follows from \cite[Lemma 8.1.10]{ambrosiogiglisavare}.
For every $c>0$, the set $\cK$ of absolutely continuous curves $\gamma$ such that $\gamma(0)\in K$ and $\int_0^T \|\gamma'(t)\|^2dt<c$ is relatively compact in $C([0,T],\R^n)$.
The tightness of $(\vartheta_\varepsilon)_\varepsilon$ follows from~(\ref{eq:tightvartheta})
and Markov's inequality.

As a consequence, there exists a measure $\vartheta$
such that $\vartheta_{\varepsilon_i}$ converges weak* to $\vartheta$, along some subsequence $(\varepsilon_i)$. The measure $\vartheta_{\varepsilon_i}$ being $\Theta$-invariant, the same holds for $\vartheta$, and the first point is proved.
Also, letting $\varepsilon_i$ tend to zero in Eq.~(\ref{eq:ev-epsilon}), we obtain that $(\text{ev}_t)_*\vartheta =\pi_*\mu$ for every $t$.

Consider a fixed $T>0$. For any smooth and compactly supported vector field
$w\in C_c^\infty(\R^n,\R^n)$, consider the map $F_w:C(\R,\R^n)\to \R_+$ given by:
\begin{equation}
    F_w(\gamma) \coloneqq \int_0^T\left\|\gamma(t)-\gamma(0)-\int_0^tw (\gamma(s) )ds\right\|dt\,.
    \label{eq:Fw}
\end{equation}
Observe that $F_{v_\varepsilon}(\gamma)=0$ for $\vartheta_\varepsilon$-a.e.~$\gamma$.

Let $\delta>0$.  From the Stone-Weierstrass theorem and the existence of $C^\infty$ bump functions, it follows that compactly-supported, $C^\infty$ functions are dense in $L^1(\pi_*\mu)$. In particular, 
there exists $w\in C_c^\infty(\R^n,\R^n)$ such that $\int \|w(x)-v_\mu(x)\|\pi_*\mu(dx)<\delta$.
Define:
$$
w_\varepsilon(x) := \rho_\varepsilon(x)^{-1} \int \psi_\varepsilon(x-y)w(y)\pi_*\mu(dy)\,.
$$
Using that $F_{v_\varepsilon}$ is equal to zero $\vartheta_\varepsilon$-a.e.,
and using the fact that $(\text{ev}_s)_*\vartheta_\varepsilon = \rho_\varepsilon(x)dx$, we obtain:
\begin{align}
&\int F_w(\gamma)\, d\vartheta_\varepsilon(\gamma) \leq  \int F_{w_\varepsilon}(\gamma)\, d\vartheta_\varepsilon(\gamma) 
+ \int \int_0^T\int_0^t\|w_\varepsilon (\gamma(s) ) - w(\gamma(s) )\|ds dt\, d\vartheta_\varepsilon(\gamma)\nonumber\\
&\leq  T^2\int\| w_\varepsilon (x)-v_\varepsilon (x)\|\rho_\varepsilon(x)dx \, 
+  T^2 \int \|w_\varepsilon (x) - w(x)\| \rho_\varepsilon(x)dx  \,. 
\label{eq:deux-termes}
\end{align}
The integral in the first term in the right-hand side of Eq.~(\ref{eq:deux-termes}) satisfies:
\begin{align*}
\int\| w_\varepsilon (x)-v_\varepsilon (x)\|\rho_\varepsilon(x)dx
&\leq \int\int \psi_\varepsilon(x-y)\|w(y)-v_\mu (y)\|\pi_*\mu(dy)dx \\
&= \int  \|w(y)-v_\mu(y)\|\pi_*\mu(dy)\leq \delta\,.    
\end{align*}
The second term in Eq.~(\ref{eq:deux-termes}) can be handled by noting that
$\|w(x)-w(y)\|\leq C\|x-y\|$ for some Lipschitz constant $C$ (which may depend on $\delta$).
By straightforward algebra,
\begin{align*}
\int\|w_\varepsilon(x)-w(x)\|\rho_\varepsilon(x)dx %
&\leq C\varepsilon \int\|x\|\psi(x)dx\,.
\end{align*}
Finally, letting $\varepsilon$ tend to zero in Eq.~(\ref{eq:deux-termes}) along the sequence $(\varepsilon_i)$, and by using the continuity of the map $F_w$, we obtain that $\int F_w d\vartheta \leq T^2\delta$. 
Define $F_{v_\mu}$ as in Eq.~(\ref{eq:Fw}), with $v_\mu$ in place of $w$. The function $F_{v_\mu}$
is well defined $\vartheta$-almost everywhere, and satisfies:
\begin{align}
\label{eq:thelefthand}
\int F_{v_\mu}(\gamma) d\vartheta(\gamma)
&\leq T^2\delta + \int\int_0^T\int_0^t \|w (\gamma(s) )-v_\mu(\gamma(s) )\|ds dt d\vartheta(\gamma) \nonumber \\
&\leq T^2 \delta + T^2 \int \|w (x)-v_\mu(x)\| \pi_*\mu(x)\,,    \nonumber
\end{align}
where the last inequality is due to the fact that $(\mathrm{ev}_s)_*\vartheta = \pi_*\mu$
for all $s$. Therefore, $\int F_{v_\mu} d\vartheta$ is bounded
by $2T^2\delta$. As $\delta$ was arbitrary, the former is equal to zero.
We conclude that for $\vartheta$-almost every $\gamma$,
$$
\forall t\in \R,\, \gamma(t)  = \gamma_0 + \int_0^t v_\mu(\gamma(s) )ds\,.
$$
This proves the second point and the last point of the theorem.
Moreover, note that for $\vartheta$-almost every $\gamma$, $\gamma'(t) = v_\mu(\gamma(t))$ for almost every $t$. Hence, the third point follows
by the $\Theta$-invariance of $\vartheta$.
\end{proof}

\subsection{Consequences}
\label{sec:consequencessuperposition}

We review some consequences of Young's superposition principle.

\begin{prop}
  \label{prop:circ-nulle}
  Let $f\colon{}\R^n\to \R$ a path-differentiable function and $\mu\in \cP(\R^n\times\R^n)$ be a closed measure. Then, for every measurable map
  $\xi\colon{}\R^n\to\R^n$ such that $\xi(x)\in \partial f(x)$ for $\pi_*\mu$-almost every $x$,
   $\int \ps{\xi(x),v}d\mu(x,v)=0$. 
 \end{prop}
\begin{proof}
  By Prop.~\ref{prop:superposition}, there exists a  measure $\vartheta$ on $\cC(\R,\R^n)$ such that:
  $$
  \int \ps{\xi(x),v}d\mu(x,v) = \int \ps{\xi(x),v_\mu(x)}\pi_*\mu(dx)= \int \ps{\xi(\gamma(0)),\gamma'(0)}\vartheta(d\gamma)\,.
  $$
  As $\vartheta$ is shift-invariant, the above value coincides with $\int \ps{\xi(\gamma(t)),\gamma'(t)}\vartheta(d\gamma)$, for every $t$.
  Thus, for every $T>0$
    $$
  \int \ps{\xi(x),v}d\mu(x,v) =  \frac 1T \int_0^T \int \ps{\xi(\gamma(t)),\gamma'(t)}\vartheta(d\gamma)dt\,.
  $$
  By Fubini's theorem, the right-hand side is equal to
  $  \frac 1 T \int (f(\gamma(T))-f(\gamma(0)))\vartheta(d\gamma)$, which is equal to zero, using the fact that
  $\vartheta$ is $\Theta$-invariant.

\end{proof}

\begin{thm}
  \label{prop:invariance}
  Consider an u.s.c. map $\sH\colon{}\R^n\rightrightarrows \R^n$ with non-empty compact convex values.
    Let $\mu\in \cP(\R^n\times\R^n)$ be a closed probability measure.
Assume that  for $\pi_*\mu$-almost every $x$,
\begin{equation}
  v_\mu(x)\in \sH(x)\,.\label{eq:v-in-H-2}
\end{equation}
  Then,  $\pi(\supp\mu)\subset\BC(\sH)$. Moreover, $v_\mu=0$, $\pi_*\mu$-a.e. on the set $\zer_s(H)\coloneqq\{x\in \zer(\sH): S_H(x)=\{t\mapsto x\}\}$ \textit{i.e.}, the set of zeroes $x$ 
  of $H$ such that any solution to (\ref{eq:di}) issued from $x$ is constant. 
\end{thm}
\begin{proof}
Define the measurable map $L\colon{}\cC(\R,\R^n)\rightrightarrows \R^n$ given by Eq.~(\ref{eq:limit-set}).
Note that $L=L\circ\Theta$. As $\vartheta$ is $\Theta$-invariant,
$(\operatorname{ev}_0,L)_*\vartheta = (\operatorname{ev}_t,L)_*\vartheta$ for every $t$, where $\operatorname{ev}_t$ stands for the projection $\operatorname{ev}_t(\gamma)=\gamma(t)$.
Moreover, for every $\gamma\in \cC(\R,\R^n)$, $d(\gamma(t),L(\gamma))\to 0$. By the dominated convergence theorem,
$$
\int 1\wedge d(\gamma(0),L(\gamma))d\vartheta(\gamma) = \lim_{t\to\infty}\int 1\wedge d(\gamma(t),L(\gamma))d\vartheta(\gamma) = 0\,.
$$
Therefore, $\vartheta$-almost every $\gamma$ satisfies $\gamma(0)\in L(\gamma)$.
Moreover, $\gamma$ is a solution to (\ref{eq:di}) on $[0,+\infty)$ by the last statement of Th.~\ref{prop:superposition}.
Therefore, $\gamma(0)$ is a recurrent point of~(\ref{eq:di}). As the distribution of $\gamma(0)$ is equal to $\pi_*\mu$,
the first point is proved. 

If, moreover, $\gamma(0)\in \zer_s(\sH)$, then $\gamma\equiv\gamma(0)$. In particular, $\gamma'(0)=0$. Using the superposition principle,
$$
0 = \int \|\gamma'(0)\|\1_{\zer_s(\sH)}(\gamma(0)) d\vartheta(\gamma) = \int \|v_\mu(x)\|\1_{\zer_s(\sH)}(x)d\pi_*\mu(x)\,.
$$
Here, we used the fact that $\zer_s(H)$ is measurable.
Indeed, if $S_I\colon{}\R^n\to \cC(\R,\R^n)$ is the map such that $S_I(x)$ is the constant function $t\mapsto x$, the set 
$\zer_s(H)$ coincides with
$\{ x\in \R^n : (S_H-S_I)(x) = \{t \mapsto 0\}\}$.
By \cite{aub-cel-(livre)84}, $S_H \colon \R^n \rightrightarrows C(\R,\R^n)$ is an upper semicontinuous map with closed graph and is therefore measurable by \cite[Th. 18.20]{charalambos2013infinite}. The same holds for $S_H-S_I$.
Thus, $\zer_s(H)=(S_H-S_I)^{-1}(t\mapsto 0)$ is also measurable.
\end{proof}
\begin{defn}[$H$-invariant measure \cite{fau-rot-13,mil-aki-99}]\label{def:invariantmeasure}
  A probability measure $\nu\in\cP(\R^n)$ is \emph{$H$-invariant (for~(\ref{eq:di}))}
  if there exists a $\Theta$-invariant measure $\vartheta\in\cP(C(\R,\R^n))$
  which is supported by the set of (complete) solutions to \eqref{eq:di}, and such that
  $\nu=\vartheta (\operatorname{ev}_0)^{-1}$, where $\operatorname{ev}_0$ is the projection $\operatorname{ev}_0(\gamma) \coloneqq  \gamma(0)$.
  \end{defn}
\begin{remark}%
   Consider a closed measure $\mu$ satisfying Eq.~(\ref{eq:v-in-H-2}), and denote by $\vartheta$
  the measure of  Th. \ref{prop:superposition}. By the last point of Th. \ref{prop:superposition}, $\vartheta$
  is supported by the set of (complete) solutions to (\ref{eq:di}). The superposition principle also
  shows that $\pi_*\mu=\vartheta (\operatorname{ev}_0)^{-1}$. As a consequence, $\pi_*\mu$ is an invariant measure to~(\ref{eq:di})
  in the sense of Faure and Roth \cite{fau-rot-13}, and the two measures $\vartheta$ represent the same object in both cases. Compare also with Corollary \ref{cor:invariance}.
\end{remark}

\begin{thm}
    \label{thm:lyap-centroid}
  Consider an u.s.c. map $\sH\colon{} \R^n\rightrightarrows  \R^n$ 
  with non-empty compact convex values. Consider the differential inclusion~\eqref{eq:di}
  and let $V\colon{}\R^n\to \R$ be a path-differentiable Lyapunov function for some $\Lambda\subset \R^n$.
  Let $\mu\in \cP(\R^n\times\R^n)$ be a closed measure satisfying Eq. \eqref{eq:v-in-H-2}
  $\pi_*\mu$-almost everywhere.
  Then, $v_\mu(x)\in \partial V(x)^\bot$ for  $\pi_*\mu$-almost every $x$.
\end{thm}

\begin{proof}
  Let $\vartheta$ be the measure given by Th.~\ref{prop:superposition}. By Prop.~\ref{prop:lyap-BC} and Th.~\ref{prop:invariance},
  for $\vartheta$-almost every $\gamma$,
  $$
  \gamma(\R)\subset\pi(\supp(\mu))\subset\BC(\sH)\subset \Lambda\,.
  $$
  Therefore, $V\circ \gamma$ is constant. As $V$ is path-differentiable, this implies that
  for almost every $t$ and for every $v\in \partial V(\gamma(t))$, $\ps{v,\gamma'(t)}=0$.
In other words, $\gamma'(t) \in \partial V(\gamma(t))^\bot$.
As $\vartheta$ is $\Theta$-invariant,  $\gamma'(0) \in \partial V(\gamma(0))^\bot$, $\vartheta$-a.e.
The result follows from Eq.~(\ref{eq:superposition}) where $\phi(x,v)$ is taken as the indicator function of $\graph(\partial V^\bot)$.
\end{proof}

\section{Stochastic approximation}
\label{sec:main}
This is the main section of the paper. In Section \ref{sec:framework}, we give the precise framework of the random sequences $(x_i)$ that interest us, we define the measures $\mu_i$. Then, in Section \ref{sec:weakaccpoints}, we study the accumulation points of the sequence $(\mu_i)$ establishing their closedness in Proposition \ref{prop:closed} and their invariance for $x'(t)=H(x(t))$ in Corollary \ref{cor:invariance}, as well as their link to the Birkhoff center.  We then proceed, in Section \ref{sec:essacc} to define and study the essential accumulation set and to link it to the Birkhoff center of $H$. Finally, we collect in Section \ref{sec:osccomp} our discussion regarding oscillation compensation properties for general differential inclusions $H$ and for those associated to Lyapunov functions $V$.

\subsection{The Framework}

\label{sec:framework}

We introduce a set-valued map $H\colon{}\R^n\rightrightarrows \R^n$, which satisfies the following assumption.
\begin{assumption}
  \label{hyp:H}
 The map $H\colon{}\R^n\rightrightarrows \R^n$ is u.s.c. with non-empty compact convex values.
\end{assumption}
We define:
\begin{equation}
    H^\delta(x) \coloneqq \left\{y\in \R^n:\exists z\text{ s.t. } \|z-x\|\leq \delta\text{ and } d(y,H(z))\leq\delta\right\}\,.
  \label{eq:Hdelta}
\end{equation}
Let $(\Omega, {\mathcal F}, (\mathcal{F}_i)_{i\in{\mathbb N}}, {\mathbb P})$ be a filtered probability space. Consider a real sequence $(\epsilon_i)\in (0,+\infty)^{\mathbb N}$. Assume that there exists a $(\mathcal{F}_i)$-adapted stochastic process $(x_i, \eta_i,\delta_i)$ on $\Omega$ to $({\mathbb R}^n\times{\mathbb R}^n\times [0,\infty))^{\mathbb N}$, equipped with the Borel $\sigma$-field associated with the product topology, such that the following inclusion holds almost everywhere:
\begin{equation}
  \label{eq:rm}
  x_{i+1} \in x_i +\epsilon_i H^{\delta_i}(x_i) + \epsilon_i\eta_{i+1}.
\end{equation}
\begin{remark}
    Examples of processes satisfying Eq.~(\ref{eq:rm}) are provided in Section~\ref{sec:app}. In the general setting, one can construct such a process by setting $\Omega\coloneqq ({\mathbb R}^n\times [0,\infty))^{\mathbb N}$, $\mathcal F$ equal to the Borel $\sigma$-field associated with the product topology, $(\eta_i,\delta_i)_i$ equal to the canonical process on $\Omega$, and $(\mathcal F_i)$ equal to the natural filtration. By Assumption~\ref{hyp:H} and the measurable selection theorem \cite[Theorem III.9]{cas-val77}, there exists a Borel map $\varphi:{\mathbb R}^n\times [0,\infty)\to{\mathbb R}^n$, such that $\varphi(x,\delta)\in H^\delta(x)$ for every $(x,\delta)$. Given an arbitrary $x_0\in {\mathbb R}^n$, the sequence of r.v. $(x_i)$ iteratively defined by $x_{i+1} = x_i + \epsilon_i \varphi(x_i,\delta_i)+\epsilon_i\eta_{i+1}$ is adapted, and satisfies the inclusion~(\ref{eq:rm}).
\end{remark}

We introduce the following event:
$$
\Gamma \coloneqq \{(x_i)\text{ is bounded and }\delta_i\to 0\}\,.
$$
\begin{remark}
The set $\Gamma$ is measurable, because it is the intersection of the measurable set $\{\limsup_i\delta_i = 0\}$ and the countable union of the sets $\{\sup_i\|x_i\|\leq q\}$ for $q\in\mathbb N$.  
\end{remark}

We make the following assumptions.
\begin{assumption}
    \label{hyp:noise}
  For every $i\in \mathbb N$, the r.v. $\|\eta_i\|$ is $\mathbb P$- integrable, and satisfies $\mathbb P$-a.e.: $$\mathbb E(\eta_{i+1}|\mathcal F_i)=0\,$$  Moreover, there exists $q>1$, such that
  \begin{equation}
    \label{eq:arc}
    \sup_i \E\left(\|\eta_{i+1}\|^q|\cF_i\right)<\infty,\,\ \ {\mathbb P}\text{-a.e. on }\Gamma\,.
  \end{equation}
\end{assumption}
\begin{assumption}
    \label{hyp:step} The step size sequence $(\epsilon_i)$ is positive and satisfies:
  \begin{enumerate}[i)]
  \item $\sum_i\epsilon_i = +\infty$.
  \item $\epsilon_i\to 0$.
  \end{enumerate}
\end{assumption}
We define the velocity  
\[v_{i+1}=\frac{x_{i+1}-x_{i}}{\epsilon_{i}} %
\]%
We introduce the random measure $\mu_i$ on $\mcB(\R^n)\otimes\mcB(\R^n)$ by
\begin{equation}
\mu_i = \frac{\sum_{j=0}^i \epsilon_j\delta_{(x_j,v_{j+1})}}{\sum_{j=0}^i \epsilon_j}\,.
\label{eq:mu}
\end{equation}
\begin{remark}
    Note that $\mu_i$ is a random variable on $\Omega$ to the space of Borel probability measures on $\R^n\times\R^n$, equipped with Borel $\sigma$-field associated with the weak* topology.
\end{remark}

\subsection{Weak* accumulation points}
\label{sec:weakaccpoints}

We start with a preliminary lemma. 
\begin{lem}
\label{lem:chow}
Let Assumptions~\ref{hyp:H}, \ref{hyp:noise} and \ref{hyp:step}-i) hold.
Define $\bar v_{i} \coloneqq v_{i}-\eta_{i}$ and define
$\bar\mu_i$ in the same way as $\mu_i$, only replacing $v_{i+1}$ by $\bar v_{i+1}$ in Eq.~(\ref{eq:mu}).
The following holds $\mathbb{P}$-a.e. on $\Gamma$:
\begin{align}
  & \sup_i \int \|v\|^qd\mu_i(x,v)<\infty
    \label{eq:lyapunov-mu-i}\\
  & \sup_i \int \|v\|^qd\bar\mu_i(x,v)<\infty
    \label{eq:lyapunov-bar-mu-i}\\
  & \lim_{i\to\infty}\frac{\sum_{j\leq i} \epsilon_j h(x_j) \eta_{j+1}}{\sum_{j\leq i}\epsilon_j} = 0\,,
  \label{eq:mtgl-eta}
\end{align}
for every measurable function $h\colon{}\R^n\to\R$ which is bounded on bounded sets.
In particular, $(\mu_i)$ and $(\bar\mu_i)$ are tight, $\P$-a.e. on $\Gamma$.
\end{lem}
\begin{proof}
  Define $\bar v_{i} \coloneqq v_i-\eta_i$. By definition, $\bar v_{i+1}\in H^{\delta_i}(x_i)$.

  We first establish Eq.~(\ref{eq:mtgl-eta}).     
    Denote by $M_{i+1}$ the numerator of the ratio in Eq.~(\ref{eq:mtgl-eta}). Note that $(M_i)$ is a martingale.
    Set $\alpha_i\coloneqq\E(\|M_{i+1}-M_i\|^q\,|\cF_i)$ and $s_i\coloneqq\sum_{j\leq i}\alpha_j$.
  We use Lemma~\ref{lem:duflo}.
  On the event $\{\sup_i s_i<\infty\}$, $M_i$ converges a.s. to a finite random variable. As $\sum_i\epsilon_i=\infty$,
  we obtain that $M_{i+1}/\sum_{j\leq i}\epsilon_j\to 0$ a.s. on that event.
  On the event $\{s_i\to\infty\}$, we have that, almost surely,
  $$
  \frac{M_{i+1}}{(s_i \log^2(s_i))^{1/q}}\to 0\,.
  $$
  Note that $\alpha_i = \epsilon_i^q |h(x_i)|^q\E( \|\eta_{i+1}\|^q\,|\cF_i)$. Thus, by Assumption~\ref{hyp:noise},
  $\alpha_i/\epsilon_i^q$ is uniformly bounded a.e. on $\Gamma$.
  This implies that, for some random variable\ $c$, which is finite a.e. on $\Gamma$,
  $$
  \frac{M_{i+1}}{\sum_{j\leq i}\epsilon_j} \leq c 
  \frac{M_{i+1}}{(s_i \log^2(s_i))^{1/q}} U_i^{1/q},,
  $$
  where 
  $$
  U_i \coloneqq \frac{\left(\sum_{j\leq i}\epsilon_j^q\right)\log^2\left(\sum_{j\leq i}\epsilon_j^q\right)}{\left(\sum_{j\leq i}\epsilon_j\right)^q}\,.
  $$
Using that $\sum_{j\leq i} \epsilon_j^q\leq \sum_{j\leq i}\epsilon_j$ for large $i$, we obtain that $U_i$ tends to zero.  
Thus,  $M_{i+1}/\sum_{j\leq i}\epsilon_j\to 0$ a.s. on the event $\{s_i\to\infty\}\cap \Gamma$.

  We now establish  Eq.~(\ref{eq:lyapunov-mu-i}). We decompose:
  \begin{equation*}
    \frac{\sum_{j\leq i} \epsilon_j \|v_{j+1}\|^q}{\sum_{j\leq i}\epsilon_j} = \frac{\sum_{j\leq i} \epsilon_j \E(\|v_{j+1}\|^q|\cF_j)}{\sum_{j\leq i}\epsilon_j} 
                                                                                     + \frac{\sum_{j\leq i} \epsilon_j \left(\|v_{j+1}\|^q- \E(\|v_{j+1}\|^q|\cF_j)\right)}{\sum_{j\leq i}\epsilon_j}\,.
  \end{equation*}
  The second term in the right-hand side of the above equality tends to zero a.s. on $\Gamma$ as a consequence of Eq.~(\ref{eq:mtgl-eta}),
  used with the martingale increment $\|v_{j+1}\|^q- \E(\|v_{j+1}\|^q|\cF_j)$ instead of $\eta_{j+1}$ (the proof is identical).
  Consider the first term. The sequence $\E(\|v_{j+1}\|^q|\cF_j)$ is uniformly bounded a.e. on $\Gamma$. 
  Consequently, Eq.~(\ref{eq:lyapunov-mu-i}) is shown. The proof of Eq.~(\ref{eq:lyapunov-bar-mu-i}) follows the same line.
\end{proof} 

Recall that $\pi\colon{}{\mathbb R}^n\times{\mathbb R}^n\to{\mathbb R}^n$ is the projection, $\pi(x,v)\coloneqq x$.
\begin{prop}
\label{prop:main}
Let Assumptions~\ref{hyp:H}, \ref{hyp:noise} and \ref{hyp:step}-i) hold.
Then, $\mathbb{P}$-a.e. $((x_i),(\delta_i))$ on $\Gamma$, any weak* accumulation point $\mu$ of $(\mu_i)$
satisfies $v_\mu(x)\in \sH(x)$ for $\pi_*\mu$-almost all $x$.
\end{prop}

\begin{proof}
  Let $\mu$ be an accumulation point of $(\mu_i)$, that is, $\mu_i\to\mu$ along some subsequence.
  Extracting a further subsequence, one can assume as well that $\bar\mu_i\to\bar\mu$ along the same subsequence,
  for some other limiting measure $\bar \mu$.

  For every bounded continuous function $h$,
  \begin{equation}
  \int h(x)v\,d\mu_i(x,v) - \int h(x)v\,d\bar\mu_i(x,v) = \frac{\sum_{j\leq i} \epsilon_j h(x_j) \eta_{j+1}}{\sum_{j\leq i}\epsilon_j} \,.\label{eq:diff-mu-barmu}
\end{equation}
  The right-hand side tends a.s. to zero on $\Gamma$ by Lem.~\ref{lem:chow}.
  Eq. (\ref{eq:lyapunov-mu-i}) and (\ref{eq:lyapunov-bar-mu-i}) of Lem.~\ref{lem:chow}
  ensure uniform integrability, and thus imply that the left-hand side of Eq.~(\ref{eq:diff-mu-barmu}) converges
  to $\int h(x)vd\mu(x,v) - \int h(x)vd\bar\mu(x,v)$.
  Thus, for every bounded continuous function $h$,
  $$\int h(x)v\,d\mu(x,v)=\int h(x)v\,d\bar\mu(x,v)\,.$$ By disintegration, $\int h(x)v_\mu(x)d\pi_*\mu(x)=\int h(x)v_{\bar\mu}(x)d\pi_*\bar\mu(x)$. As $\pi_*\mu=\pi_*\bar\mu$,
  we conclude that $v_\mu=v_{\bar\mu}$ $\pi_*\mu$-a.e.
  
  Let $\delta>0$ be arbitrary. As $\delta_i\to 0$ a.e. on $\Gamma$,
  $\bar v_{i+1}\in H^\delta(x_i)$ for every $i$ large enough. Equivalently, $\1_{\graph H^\delta}(x_i,\bar v_{i+1})=1$ for large $i$. Thus,
  $$
  \int \1_{\graph H^\delta}(x,v)d\bar \mu_i(x,v)\to 1\,.
  $$
  As $\graph H^\delta$ is closed, $\int \1_{\graph H^\delta}(x,v)d\bar \mu(x,v)=1$. Desintegrate $\bar\mu = \int \bar\mu_x d\pi_*\mu(x)$.
  By Fubini's theorem, $\int \1_{\graph H^\delta}(x,v)d\bar \mu_x(v)=1$ for $\pi_*\mu$-a.e $x$. This proves that
  $\supp(\mu_x)\subset H^\delta(x)$. Using that $H^\delta(x)$ is convex, it follows that $v_{\bar\mu}(x)\in H^\delta(x)$ for $\pi_*\mu$-a.e. $x$.
  As $\delta$ is arbitrary, $v_{\bar\mu}(x)\in \bigcap_{\delta>0}H^\delta(x)=H(x)$. As $v_\mu$ coincides with $v_{\bar \mu}$ $\pi_*$-a.e., the conclusion follows.
\end{proof}

 \begin{prop}
   \label{prop:closed}
Let Assumptions~\ref{hyp:H}, \ref{hyp:noise} and \ref{hyp:step} hold.
Then, $\mathbb{P}$-a.e. $((x_i),(\delta_i))$ on $\Gamma$,
any accumulation point $\mu$ of $(\mu_i)$ satisfies, for every $g\in \cC^\infty(\R^n,\R)$,
$
\int \ps{\nabla g(x),v}d\mu(x,v)=0\,.
$
 \end{prop}
 \begin{proof}
   Consider a fixed $\omega\in \Gamma$ (we omit the dependency in $\omega$)
   such that $(x_i)$ is bounded and $\int \|v\|^qd\mu_i(x,v)$ is uniformy bounded.
   Let $\mu$ be an accumulation point of $(\mu_i)$.
   Consider $g\in \cC^\infty(\R^n,\R)$. By uniform integrability,
   $$
   \int \ps{\nabla g(x),v}d\mu(x,v)=\lim_{i\to\infty}\int \ps{\nabla g(x),v}d\mu_i(x,v)\,.
   $$
   Set $t_i\coloneqq \sum_{j=0}^{i-1}\epsilon_j$ for every $i\geq 0$. Equivalently, $\epsilon_i = t_{i+1}-t_{i}$.
   Let $\sx\colon{}\R_+\to\R^n$ be the piecewise linear interpolated process associated with $(x_i)$, defined by:
   $$
   \sx(t) \coloneqq x_i + v_{i+1}(t-t_i)
   $$
   for every $t\in [t_i,t_{i+1})$. For every $i$, define the measure $\tilde\mu_i\in \cP(\R^n\times\R^n)$ given by:
   $$
   \tilde\mu_i \coloneqq  \frac 1{t_{i+1}}\int^{t_{i+1}}_{0}\delta_{(\sx(t),\sx'(t))}dt\,.
   $$
   Define
   $\varphi(x,v) \coloneqq \ps{\nabla g(x),v}$. After straightforward algebra,
   $$
   \mu_i(\varphi) - \tilde\mu_i(\varphi) = \frac 1{t_{i+1}}\sum_{j=0}^i\int^{t_{j+1}}_{t_{j}}\left(\varphi(x_j,v_{j+1})-\varphi(\sx(t),v_{j+1})\right)dt\,.
   $$
   Note that $\|\nabla g(x)\|$ is Lipschitz continuous and bounded on $K$. Thus, there exists a constant $C>0$ such that:
      \begin{align*}
        |\mu_i(\varphi) - \tilde\mu_i(\varphi)| &\leq \frac 1{t_{i+1}}\sum_{j=0}^i\int^{t_{j+1}}_{t_{j}}\|\nabla g(x_j)-\nabla g(\sx(t))\|\|v_{j+1}\|dt\\
                                                &\leq \frac C{t_{i+1}}\sum_{j=0}^i \|v_{j+1}\|\int^{t_{j+1}}_{t_{j}}(1\wedge\|x_j-\sx(t)\|)dt\\
                &\leq \frac C{t_{i+1}}\sum_{j=0}^i \epsilon_j\|v_{j+1}\| (1\wedge\epsilon_j\|v_{j+1}\|)\,.
   \end{align*}
   For every $\varepsilon>0$, one has $\epsilon_i\leq \varepsilon$ for every $i$ large enough. Taking limits,
   $$
   \lim\sup_{i\to\infty}   |\mu_i(\varphi) - \tilde\mu_i(\varphi)| \leq C \int \|v\| (1\wedge\varepsilon\|v\|)d\mu(x,v)\,.
   $$
   As $\varepsilon$ is arbitrary, and using that $\int\|v\|d\mu(x,v)<\infty$, we obtain $|\mu_i(\varphi) - \tilde\mu_i(\varphi)|\to 0$.
   Moreover,
   $$
   \tilde\mu_i(\varphi) = \frac 1{t_{i+1}}\int_0^{t_{i+1}} \ps{\nabla g(\sx(t)),\sx'(t)}dt = \frac{g(x_{i+1})-g(x_0)}{t_{i+1}}\,.
   $$
   Thus, $\tilde\mu_i(\varphi)\to 0$, and the result is proved.
 \end{proof}

 \begin{coro}\label{cor:invariance}
     Let Assumptions~\ref{hyp:H}, \ref{hyp:noise} and \ref{hyp:step} hold.
Then, $\mathbb{P}$-a.e. $((x_i),(\delta_i))$ on $\Gamma$,
any accumulation point $\mu$ of $(\mu_i)$ is $H$-invariant %
in the sense of Definition \ref{def:invariantmeasure}.
 \end{coro}
 \begin{proof}
     This follows immediately from the superposition principle, Th. \ref{prop:superposition} (which may be applied by Proposition \ref{prop:closed}), together with Prop. \ref{prop:main}.
 \end{proof}

We recall the definition of $\zer_s(H)$ in Th.~\ref{prop:invariance}.
\begin{thm}
\label{th:main}
Let Assumptions~\ref{hyp:H}, \ref{hyp:noise} and \ref{hyp:step} hold.
Then, $\mathbb{P}$-a.e. on $\Gamma$, any accumulation point $\mu$ of $(\mu_i)$ is such that:
  \begin{enumerate}[i)]
  \item $\pi(\supp\mu)\subset\BC(\sH)$.
  \item $v_\mu=0$, $\pi_*\mu$-a.e.\ on $\zer_s(\sH)$.
  \end{enumerate}
\end{thm}
\begin{proof}
 Let $\mu$ be chosen as in Prop.~\ref{prop:closed}. Define $\mu'\in \cP(\R^n\times\R^n)$ as:
 $$
 \mu'(\varphi) = \int \varphi(x,v_\mu(x))d\pi_*\mu(x)\,,
 $$
 for every $\varphi\in \cC_b(\R^n\times\R^n,\R)$. By Prop.~\ref{prop:closed},
 $\int \ps{\nabla g(x),v}d\mu'(x,v)=0$ for every $g\in \cC^\infty(\R^n,\R)$.
  As $\sH$ is locally bounded, $\mu'$ is compactly supported.
 In sum, we conclude that $\mu'$ is closed.
 By Prop.~\ref{prop:main}, $v_\mu(x)\in \sH(x)$ for $\pi_*\mu$-almost every $x$.
 As $v_\mu = v_{\mu'}$, the conclusion of Th.~\ref{th:main} follows from
 Prop.~\ref{prop:invariance}.
\end{proof}

\subsection{Essential accumulation set}\label{sec:essacc}

Following the exposition of \cite{bolte2020long}, we introduce the notion of essential accumulation points of the sequence $(x_i)$.
For every $U\subset\R^n$, define:
$$
\tau_i^U\coloneqq \mu_i(U\times \R^n) = \frac{\sum_{j=0}^i \epsilon_j \1_U(x_j)}{\sum_{j=0}^i \epsilon_j}.
$$
The quantity $\tau_i^U$ should be interpreted as the average amount of time during which the iterates
are in the set $U$, up to iteration $i$.
\begin{defn}\label{def:essacc}
  For every $\omega\in \Omega$, a point $x$ is an \emph{essential accumulation} point of the sequence $(x_i(\omega))$ if for every neighborhood $U$ of $x$,
  $\lim\sup_i \tau_i^U(\omega)>0\,.$
\end{defn}
We denote by $\essacc{(x_i)}$ the set of essential accumulation points.
Th.~\ref{th:main} has the following consequence.
\begin{coro}\label{cor:recurrent}
Let Assumptions~\ref{hyp:H}, \ref{hyp:noise} and \ref{hyp:step} hold.
Then, $\mathbb{P}$-a.e. on $\Gamma$, %
$$
\essacc{(x_i)} \subset \BC(H)\,.
$$
In particular, if (\ref{eq:di}) admits a Lyapunov function $V\colon{}\R^n\to\R$ for some $\Lambda\subset \R^n$, then $\essacc{(x_i)} \subset\Lambda$.
\end{coro}
\begin{proof}
  The last statement is immediate from Prop.~\ref{prop:lyap-BC}. We show the first statement.
The proof follows from Th.~\ref{th:main} along with the following equality:
\begin{equation}
\essacc{(x_i)} = \cl \bigcup_{\mu} \pi(\supp \mu)\,,\label{eq:essacc-in-supp}
\end{equation}
where the union is taken over all weak* accumulation points of $(\mu_i)$.
This point is established in \cite{bolte2020long}. We provide a shorter proof, for completeness.
Consider $x\in \essacc{(x_i)}$. Choose $k>0$ and set $U$ equal to the closed ball of center $x$ and radius $1/k$.
It holds that $\limsup_i\pi_*\mu_i(U)>0$. Thus there exists $\varepsilon>0$, a closed measure $\mu$ and a subsequence $(\mu_{\varphi_i})$ s.t.
$\mu_{\varphi_i}\to\mu$ and $\pi_*\mu_i(U)\to \varepsilon$. As $U$ is closed, $\pi_*\mu(U)\geq \varepsilon>0$ by the Portmanteau Theorem.
Thus, $\supp(\pi_*\mu)\cap U\neq \emptyset$. There exists $y_k\in \pi(\supp(\mu))$ s.t. $\|y_k-x\|\leq 1/k$. This proves the first inclusion.

Conversely, consider $x$ in the right-hand side of~(\ref{eq:essacc-in-supp}), and let $U$ be an open neighborhood of $x$.
There exists a point $y\in U$ and an accumulation point $\mu$ of $(\mu_i)$ s.t. $y\in \supp(\pi_*\mu)$.
As $U$ is an open neighborhood of $y$, by definition of the support of a measure, it holds that $\pi_*\mu(U)>0$.
There exists a subsequence $(\mu_{\varphi_i})$ s.t. $\mu_{\varphi_i}\to\mu$. As $U$ is open,
$\lim\inf_i \pi_*\mu_{\varphi_i}(U)\geq \pi_*\mu(U)>0$. Thus, $\lim\sup_i \pi_*\mu_i(U)>0$, which proves
that $x$ is an essential accumulation point.
\end{proof}

\subsection{Oscillation compensation}
\label{sec:osccomp}

\begin{defn}\label{def:osccomp}
   We say that $(x_i)$ has the \emph{oscillation compensation} property if for every $\psi\in \cC_b(\R^n,\R)$,
    \begin{equation}\label{eq:osccompdef2}
     \lim_{i\to\infty}\frac {\sum_{j\leq i}\epsilon_j \psi(x_j) v_{j+1} }{\sum_{j\leq i}\epsilon_j} = 0\,.
     \end{equation}
\end{defn}

As explained %
in the introduction, the oscillation compensation property implies a slow down of the sequence $(x_i)$.
Alternatively, it is not difficult to show that, $\mathbb{P}$-a.e. on $\Gamma$,  $(x_i)$ has the oscillation compensation property if, and only if:
    \begin{equation}\label{eq:osccompdef1}
     \lim_{k\to\infty}\frac {\sum_{j\leq i_k}\epsilon_j \psi(x_j) v_{j+1} }{\sum_{j\leq i_k}\epsilon_j\psi(x_j)} = 0\,.
    \end{equation}
  for every $\psi\in \cC_b(\R^n,\R_+)$ and every subsequence $(i_k)$ such that
  \begin{equation}\label{eq:subseq-cond}
   \liminf_{k\to+\infty}\frac{\sum_{j\leq i_k}\epsilon_j\psi(x_j)}{\sum_{j\leq i_k}\epsilon_j}>0.
  \end{equation}

\begin{coro}\label{cor:compens}
  Let Assumptions~\ref{hyp:H}, \ref{hyp:noise} and \ref{hyp:step} hold. Assume that the following condition holds:
  \begin{enumerate}[i)]
  \item $\BC(\sH)=\zer_s(\sH)\,,$
  \end{enumerate}
  or the stronger condition:
  \begin{enumerate}[i)]
    \setcounter{enumi}{1}
    \item There exists a path-differentiable Lyapunov function $V\colon{}\R^n\to\R$ for some $\Lambda\subset\R^n$, such that $\sH(x)\cap \partial V(x)^\bot=\{0\}$
    for all $x\in \Lambda$.
  \end{enumerate}
Then, $\mathbb{P}$-a.e. on $\Gamma$, $(x_i)$ has the oscillation compensation property.
\end{coro}
\begin{proof}
  We first establish that the second condition implies the first one. The second condition implies that $\Lambda \subset \zer(\sH)$.
  Moreover, consider  $\gamma\in S_\sH(x)$ for an arbitrary $x\in \Lambda$. By Def.~\ref{def:lyapunov}, $\gamma([0,+\infty))\in \Lambda$,
  and $V\circ \gamma$ is constant. 
  By the path-differentiability, $0 = \ps{v,\gamma'(t)}$ for all $v\in \partial V(\gamma(t))$ and for almost all $t$.
  Thus, $\gamma'(t)$ is orthogonal to $\partial V(\gamma(t))$. As $\gamma(t)\in \Lambda$, we obtain that $\gamma'(t)=0$.
  Thus, $\gamma(t) = x$ for all $x$. This proves that $x\in \zer_s(\sH)$. Thus, $\BC(H)\subset\Lambda\subset \zer_s(\sH)$.

  Now consider the case where the first condition holds.
  By Th.~\ref{th:main}, it holds that every accumulation point $\mu$ of $(\mu_i)$ satisfies $v_\mu=0$ $\pi_*\mu$-a.e.
  This directly implies the statement.
\end{proof}

\section{Applications}
\label{sec:app}

\subsection{Stochastic subgradient descent}
\label{sec:subgradient}

The following result generalizes results of \cite{bolte2020long} to the stochastic setting. 
\begin{prop}
  \label{prop:sgd}
  Let $f\colon{}\R^n\to\R$ be a path-differentiable function. 
  Consider a random sequence $(x_i)$ satisfying:
  $$
x_{i+1} \in x_i - \epsilon_i \partial f(x_i) + \epsilon_i \eta_{i+1}\,.
  $$
  Let Assumptions~\ref{hyp:noise}, \ref{hyp:step} hold, with $\delta_i\coloneqq 0$.
The following holds a.e. on the event $\{\sup_i\|x_i\|<\infty\}$:
  \begin{enumerate}[i)]
  \item $\essacc (x_i)\subset \{x\in \R^n\,:\,0\in \partial f(x)\}$.
  \item $(x_i)$ has the oscillation compensation property. 
  \end{enumerate}
\end{prop}
The reason why we require the function $f$ to be path-differentiable is to ensure that $f$ is a Lyapunov function for its (negative) Clarke subdifferential $-\partial f$, and for 
$\Lambda$ equal to the Clarke critical set. In contrast, weaker assumptions, like for example $f$ Lipschitz, can potentially yield dynamics that are too unwieldy. In addition, the path-differentiability assumption is weak enough to cover most cases of interest in applications; see \cite{bolte2020long} for a discussion.
\begin{proof}
  Observing that $V\coloneqq f$ is a path-differentiable Lyapunov function for the differential inclusion 
  $x'(t)\in -\partial f(x(t))$, the conclusion follows from Cor.~\ref{cor:recurrent} and Cor.~\ref{cor:compens}.
\end{proof}
\begin{remark}
  More generally, the subdifferential $\partial f$ can be replaced by any \emph{conservative field}
  $H\colon{}\R^n\rightrightarrows\R^n$ (see \cite{bolte2020conservative} for definitions),
  hence extending the results to other algorithms of the flavour of the stochastic subgradient descent,
  such as the celebrated back-propagation
  algorithm.
\end{remark}

\subsection{Stochastic heavy ball}
\label{sec:heavyball}

Let $f\colon{}\R^m\to\R$ be a path-differentiable function, for some integer $m\geq 1$. We consider a random sequence $(x_i)$
on $\R^m\times\R^m$, satisfying $x_i=(q_i,p_i)$, where:
\begin{equation}
  \label{eq:shb}
  \begin{array}[h]{ll}
    &q_{i+1} = q_i + \alpha_i p_{i+1}\,, \\
    &p_{i+1} \in  (1-\beta_i) p_i -\beta_i \partial f(q_i) + \beta_i \eta_{i+1}\,,
  \end{array}
\end{equation}
where $(\alpha_i)$, $(\beta_i)$ are two sequences of positive step sizes.
This method is sometimes refered to as Stochastic Heavy Ball (SHB) \cite{gadat2018stochastic,barakat2021stochastic}. It can be rewritten as a function of the variable $(q_i)$ in the following way:
$$
q_{i+1} \in q_i + \beta_i'\left(-\partial f(q_i)+ \eta_{i+1}\right) + \alpha_i' (q_i-q_{i-1})\,,
$$
with the change of variable $\alpha_i' \coloneqq \alpha_i \alpha_{i-1}^{-1}(1-\beta_i)$
and $\beta_i' \coloneqq \alpha_i\beta_i$.
We make the following assumption:
\begin{assumption}
  \label{hyp:step-shb}
The sequences $(\alpha_i)$, $(\beta_i)$ are positive, and satisfy $\sum\alpha_i=+\infty$ and $\alpha_i\to 0$. Moreover, there exists $c>0$ s.t. $\alpha_i/\beta_i\to c$.
\end{assumption}
\begin{prop}
  Let $f\colon{}\R^m\to\R$ be a path-differentiable function. 
  Consider the sequence $(x_i)$, $x_i=(q_i,p_i)$ given by Eq.~(\ref{eq:shb}). Let Assumptions~\ref{hyp:noise} and \ref{hyp:step-shb} hold.
  Then,
  \begin{enumerate}[i)]
  \item $\essacc(x_i)\subset \{(x,0)\,:\,0\in \partial f(x)\}$.
  \item $(x_i)$ has the oscillation compensation property. 
  \end{enumerate}
\end{prop}
\begin{proof}
  In order to keep the exposition concise, assume
  $\alpha_i = c\beta_i$. %
  The iterates $(x_i)$ have the form~(\ref{eq:rm}), where
  $\delta_i=0$, where $\eta_{i+1}$ is replaced with $(0,\eta_{i+1})$,
  and where the map $H\colon{}\R^m\times\R^m\rightrightarrows\R^m\times\R^m$
  is given by:
$$
H(q,p) \coloneqq (-cp, \partial f (q)-p)\,.
$$
Considering (\ref{eq:di}), the function  $V\colon{}\R^m\times\R^m\to\R$ given by
$
V(q,p) \coloneqq f(q)+c\frac{\|p\|^2}2
$, is a path differentiable Lyapunov function for the set $\R^m\times \{0\}$.
Indeed, any solution $\gamma(t) = (q(t),p(t))$ to (\ref{eq:di}) satisfies, for almost every $t$,
\begin{equation}
(V\circ \gamma)'(t) = -c\|p(t)\|^2\,.\label{eq:Vshb}
\end{equation}
Thus, $\BC(\sH)\subset \R^m\times \{0\}$.
Consider an arbitrary recurrent point $x$ (if it exists). Then, $x=(q,0)$ for some $q\in \R^m$.
Consider $\gamma\in S_H(x)$ s.t. $x\in L(\gamma)$. Set $\gamma(t) = (q(t),p(t))$. As the function $V\circ \gamma$ is constant,
Eq.~(\ref{eq:Vshb}) implies that $p(t)=0$ a.e. As $q'(t) = -cp(t)=0$, we conclude that $q(t) = x$ for all $t$.
Lastly, the inclusion $p'(t) \in \partial f(q(t))-p(t)$ reads $0\in \partial f(x)$. Thus $x\in \zer(-\partial f)$.
This proves that $\BC(\sH) \subset \{(x,0)\,:\,0\in \partial f(x)\}$. As $\gamma$ is constant, this also
proves that $\BC(H) = \zer_s(\sH)$.
The conclusion follows from Cor.~\ref{cor:recurrent} and Cor.~\ref{cor:compens}.
\end{proof}

\subsection{Fictitious play and best-response game dynamics}
\label{sec:games}

There are many possible dynamics for games; see for example \cite{hofbauer2011deterministic}. The setting in this section roughly corresponds to the one developed in \cite{ben-hof-sor-05,ben-hof-sor-(partII)06}. 
We consider the dynamics of fictitious play, which in this case means that the players play simultaneously at each turn, and they choose their respective actions with knowledge of the sequence $(\xi_n)$ of averages of the actions of a game (to be specified below). The dynamics of interest will be that of the sequence $(\xi_n)$.

Consider a game with $m$ players and, for each $i=1,\dots,m$, let $X^i$ be a compact, convex subsets of $\R^n$ that will play the role of the possible (mixed) actions of player $i$. Typically, $X^i$ is a simplex, so that there are finitely many pure actions available to player $i$, corresponding to the vertices of the simplex.
Here are some useful notations: we will denote the set of actions available to the opponents of player $i$ by
\[X^{-i}=X^1\times X^2\times \dots\times X^{i-1}\times X^{i+1}\times \dots\times X^m\subset\R^{n(m-1)},\]
and for $x=(x^1,\dots,x^m)\in X^1\times\dots\times X^m$,
\[x^{-i}=(x^1,x^2,\dots,x^{i-1},x^{i+1},\dots,x^m)\in X^{-i}.\]
Each player also has a payoff function $U^i\colon X^i\times X^{-i}\to\R$. Player $i$ wants to maximize her payoff $U^i$, so for each given vector $\xi^{-i}\in X^{-i}$ of actions of its opponents, it makes sense to define the \emph{best response set} by
\[BR^i(\xi^{-i})=\argmax_{\xi^i\in X^i}U^i(\xi^i,\xi^{-i})\subseteq X^i,\]
and we will assume that the functions $U^i$ are such that each $BR^i(\xi^{-i})$ is convex and compact.
It follows from the Berge maximum theorem that, if $U^i$ is continuous, then also the maps $BR^i$ are upper semicontinuous.

Let $x_\ell=(x^1_\ell,\dots,x^m_\ell)\in X^1\times \dots\times X^m$ be the profile of actions taken by all players at stage $\ell$.
We denote the average of the actions up to stage $n$ by
\begin{equation}\label{eq:xiaverage}
 \xi_n=\frac1n\sum_{\ell=0}^nx_\ell
\end{equation}
Player $i$ chooses, at each stage $n+1$, a point 
\begin{equation}\label{eq:fictplay}
 x_{n+1}^i\in BR^i(\xi^{-i}_n).
\end{equation}
In the game-theoretical literature, this is known as the \emph{fictitious play assumption}.

We let $H\colon X^1\times\dots\times X^m\rightrightarrows X^1\times\dots\times X^m$ be the map given by 
\begin{equation}\label{eq:Hgames}
 H(\xi^1,\dots,\xi^m)=(BR^1(\xi^{-1})-\xi^{1})\times\dots\times(BR^m(\xi^{-m})-\xi^{m}).
\end{equation}
Then $H(\xi_n)$ is a nonempty, convex and compact set (hence consistent with Assumption \ref{hyp:H}) that contains $x_{n+1}-\xi_n$.

We have, by \eqref{eq:fictplay},%
\[\xi_{n+1}-\xi_n\in \frac1{n+1}H(\xi_n)%
,\]
and the sequence $(\xi_n,\eta_n,\epsilon_n)_{n\geq1}$, with $\eta_n=0$ and $\epsilon_n=1/n$, satisfies \eqref{eq:rm} and Assumptions \ref{hyp:noise} and \ref{hyp:step}, and the results of Th. \ref{th:main} hold. 

A \emph{Nash equilibrium} is a point $\xi\in X^1\times\cdots\times X^m$ such that $U(\xi)\geq U(x^{i},\xi^{-i})$ for all $x^i\in X^i$ and $i=1,\dots,m$; in other words, it is a point such that $\xi^i\in BR^i(\xi^{-i})$ for all $i$. Observe that this is equivalent to having $0\in H(\xi)$, so the constant solutions $\xi_0=\xi_{1}=\dots$ are precisely the Nash equilibria.

We also have:

\begin{prop}\label{prop:generalgames}
 The essential  accumulation set of the sequence $(\xi_i)$ is contained in the Birkhoff center $\BC(H)$, and satisfies the oscillation compensation property on the stable zeros $\zer_s(H)$.
\end{prop}

\begin{remark}
 There are a few types of games that admit Lyapunov functions, among which we count zero-sum and potential games; see for example \cite{hofbauersorin} and \cite[Section 5.3]{ben-hof-sor-05}, respectively. In these cases, Theorem \ref{thm:lyap-centroid} and Corollary \ref{cor:compens} (with the second hypothesis) apply, and thus recover the results that were obtained before for example in \cite{ben-hof-sor-05,gaunersdorfer1995fictitious,hofbauersorin}. Although in many of these games the orbits of fictitious play converge to the Nash equilibrium, there are some examples in which this is not the case; notably, in the generalized rock-paper-scissors game the orbits accumulate around the so-called Shapely polygons and they never converge \cite{gaunersdorfer1995fictitious}. 
\end{remark}

\begin{remark}\label{rmk:moregeneralgames}
It is possible to add a noise and to exchange the condition \eqref{eq:fictplay} %
 with the assumption that the expectation of $x^i$ falls in $BR^i(\xi^{-i})$, and still get a statement like the one in Proposition \ref{prop:generalgames}.
\end{remark}

\appendix
\section{A lemma on martingales}

The following lemma is a consequence of Chow's martingale theorem (see \cite[Theorem 2.17]{hall2014martingale}).
The proof can be found in \cite[Th.1.III.11]{duflo1996algorithmes}, and is partly reproduced here, for completeness.
\begin{lem}
\label{lem:duflo}
    Let $M_i = \sum_{j=1}^i X_i$ be a $\R^n$-valued $(\mathcal{F}_i)$-adapted martingale.
    Consider $1\leq q\leq 2$ and $r>0$.  Define $s_i \coloneqq \sum_{j\leq i}\mathbb E(\|X_{i+1}\|^q) |\mathcal{F}_i)$.
    Then, there exists a random variable $M_\infty$, finite a.e., such that $M_i\to M_\infty$ a.s.
    on the set $\{\lim_i s_i<\infty\}$. 
    Moreover, on the set $\{\lim_i s_i=\infty\}$, $$
    \frac{M_i}{(s_{i-1}\log^{1+r}(s_{i-1}))^{1/q}}\to 0\ a.s.
    $$
\end{lem}
\begin{proof}
The first point of the Lemma is a consequence of Chow's theorem (see \cite[Theorem 2.17]{hall2014martingale}). Therefore, we only establish the second point, following the arguments of
\cite[Th.1.III.11]{duflo1996algorithmes}. We introduce the martingale $N_i = \sum_{j\leq i}Y_i$ where $Y_i\coloneqq (s_{i-1}\log^{1+r}(s_{i-1}))^{-1/q} X_i$. We observe that 
${\mathbb E}(\|Y_{i+1}\|^q|\mathcal{F}_i)=\alpha_j(s_{i-1}\log^{1+r}(s_{i-1}))^{-1}$ is summable, because $\int_e^\infty (t\log^{1+\beta}(t))^{-1}dt=\int_1^\infty t^{-(1+r)}dt <\infty$. By the first point, there exists an almost surely finite random variable $N_\infty$ such that $N_i\to N_\infty$. The conclusion follows from the Kronecker lemma \cite[Section 2.6, pp.31]{hall2014martingale}.
\end{proof}

\section*{Acknowledgements}
The authors are very grateful for numerous enlightening discussions with J\'er\^ome Bolte and Edouard Pauwels. This line of research grew out of inspiring conversations with Sylvain Sorin.
The second author wishes to acknowledge the generous support of ANR-3IA Artificial and Natural Intelligence Toulouse Institute, and of the Air Force Office of
Scientific Research, Air Force Material Command, \textsc{usaf}, under grant number FA9550-19-1-7026.

\def\cprime{$'$} \def\cdprime{$''$} \def\cprime{$'$}

\end{document}